\documentclass[12pt]{article}
\usepackage{amsmath,amssymb,amsthm}
\usepackage{url}
\usepackage[arrow,matrix,curve]{xy}
\newtheorem{thm}{Theorem}
\newtheorem{lem}{Lemma}
\newtheorem{prop}{Proposition}
\newtheorem{cor}{Corollary}
\newtheorem*{thm*}{Theorem}
\newtheorem*{prop*}{Proposition}
\newtheorem*{cor*}{Corollary}
\newtheorem*{conj*}{Conjecture}
\theoremstyle{definition}

\theoremstyle{definition}
\newtheorem{rmk}{Remark}
\newtheorem*{rmk*}{Remark}
\theoremstyle{definition}

\newcommand{\ZZ}{\mathbb Z}
\newcommand{\RR}{\mathbb R}
\newcommand{\QQ}{\mathbb Q}
\newcommand{\CC}{\mathbb C}

\newcommand{\NN}{\mathbb N}

\def\spr{\mathop{\mathrm{span_\RR}}\nolimits}

\def\tr{\mathop{\mathrm{tr}}\nolimits}
\def\diag{\mathop{\mathrm{diag}}\nolimits}
\def\dis{\mathop{\mathrm{dis}}\nolimits}
\def\sign{\mathop{\mathrm{sign}}\nolimits}
\newcommand{\bil}{\mathcal B}
\newcommand{\qu}{\mathcal Q}
\begin{document}

\author{Hanno von Bodecker\footnote{Fakult{\"a}t f{\"u}r Mathematik, Universit{\"a}t Bielefeld, Germany}}

\title{Twisted Dirac operators on certain nilmanifolds associated to even lattices}

\date{}

\maketitle

\begin{abstract}
Starting from an even definite lattice, we construct a principal circle bundle covered by a certain three-step nilpotent Lie group $G$. On the base space, which is  again a nilmanifold, we then study the Dirac operator twisted by the associated complex line bundles. Noting that the whole situation fibers over the circle, we are able to determine the reduced $\eta$--invariant of these Dirac operators in the adiabatic limit.\\
As an application, we consider the total space of the circle bundle, equipped with a parallelism induced by $G$, as an element in the stable homotopy groups of the sphere and use the $\eta$--invariants to analyze its status in the Adams--Novikov spectral sequence. 
\end{abstract}

\section{Introduction and statement of the results}
Recall that on a Riemannian manifold $M$ with a fixed spin structure, there is a natural first order differential operator, the Dirac operator $\eth$. If $M$ is closed and odd-dimensional, the spectrum of this operator is real and discrete. An interesting invariant is the $\eta$--invariant: The sum
$$\eta(\eth,s)=\sum_{\sigma\in \mathrm{spec}(\eth)\backslash\{0\}} \sigma/|\sigma|^{1+s}$$
converges for sufficiently large real part of $s$; by analytic continuation, the expression above defines a meromorphic function on the complex plane which is actually holomorphic at $s=0$, and its value there is referred to as the $\eta$--invariant of the Dirac operator. The same holds true in the situation of the Dirac operator twisted by a hermitian vector bundle $E$ with unitary connection.

Unfortunately, only in very few situations one can hope to determine the spectrum of Dirac operators explicitly;  relatively simple exceptions come from homogeneous spaces for {\em compact} Lie groups $G$ -- in this situation, the space of $L^{2}$--spinors on $G/K$ (possibly twisted by homogeneous vector bundles) decomposes into {\em finite--dimensional} $G$--modules (the so-called isotypic components), on which the problem reduces to linear algebra (keep in mind though, that the precise decomposition into isotypical components requires knowledge of the branching rules). (As shown in \cite{Goette:1999cq}, the computation of the $\eta$--invariant in this situation can actually be carried out without explicitly determining the spectrum.)

For non-compact Lie groups $G$, things become more complicated due to the fact that the irreducible representations no longer need to be finite-dimensional. Nevertheless, for nilpotent Lie groups, there is still a well-understood representation theory, see e.g.\ \cite{Corwin:1990ve}. In \cite{Deninger:1984zt}, Deninger and Singhof have determined the $\eta$--invariant of (a slightly modified version of) the untwisted Dirac operator on Heisenberg nilmanifolds (i.e.\ quotients $M^{2n+1}=\Gamma\backslash N^{2n+1}$, where $N^{2n+1}$ denotes the two-step nilpotent, 1--connected Heisenberg group of real dimension $2n+1$, and $\Gamma\subset N^{2n+1}$ a suitable lattice). (A detailed analysis of the spectrum of the unmodified Dirac operator in the 3--dimensional case has been carried out in \cite{Ammann:1998fk}.)

The first purpose of this article is to generalize this result by computing the $\eta$--invariant of the Dirac operator twisted by a suitable hermitian line bundle on a family of nilmanifolds $M^{2r+1}$. More precisely, we proceed as follows: Given a positive-definite even lattice $\Lambda$ of rank $r$, we define a nilpotent Lie group $G_{\Lambda}$ by equipping $(\ZZ\times\Lambda\times\ZZ\times\Lambda)\otimes_{\ZZ}\RR$ with the multiplication
\begin{equation*}
\begin{split}
&(z, c,b, a)\cdot(z',c',b', a')\\
=&\left(z+z'+\bil(a, c')+\qu( a)b', c+ c'+ ab',b+b', a+ a'\right),
\end{split}
\end{equation*}
where $\qu$ and $\bil$ denote the quadratic and symmetric bilinear form, respectively. Then it is easy to see that the compact nilmanifold $\Gamma_{\Lambda}\backslash G_{\Lambda}$, where $\Gamma_{\Lambda}\subset G_{\Lambda}$ is the discrete subgroup determined by the integral structure,  carries the structure of a (right) principal circle bundle, its base being itself a nilmanifold,  denoted $M_{\Lambda}$. Explicitly specifying a parallelism for $M_{\Lambda}$ gives rise to a Riemannian metric and a preferred spin structure, and we may form the Dirac operator twisted by (powers of) the tautological hermitian line bundle $\pi\colon\lambda\rightarrow M_{\Lambda}$. In order to compute the $\eta$--invariant, we do not attempt to determine the spectrum of this operator explicitly; instead, we observe that $M_{\Lambda}$ fibers over the circle  and investigate the adiabatic limit \cite{Bismut:1989dk,Dai:1991os}, i.e.\ we study the behaviour under blowing up the metric in the direction of the base, the result being:
\begin{thm}
Let $\lambda$ be the hermitian line bundle constructed above, and let $\eth_{\epsilon}$ be the Dirac operator on $M_{\Lambda}$ with respect to the metric $\langle\cdot,\cdot\rangle_{\epsilon}$ and twisted by $\lambda^{d}$ for some non-zero integer $d$. Then the adiabatic limit of its reduced $\eta$--invariant is given by
\begin{equation*}
\lim_{\epsilon\rightarrow0}\frac{\eta(\eth_{\epsilon})+\dim\ker(\eth_{\epsilon})}{2}\equiv(\sign(d))^{r}\sum_{\rho\in\dis\Lambda_{d}}\left(\frac{1}{2}-\overline{\qu}_{d}(\rho)\right)\mod\ZZ.
\end{equation*}
\end{thm}

Recall that a parallelism on a closed manifold canonically induces a trivialization of the stable normal bundle (unique up to homotopy), and hence, by the Pontrjagin--Thom construction, determines an element in the stable homotopy groups of the sphere.
A powerful algebraic tool for studying the stable homotopy groups of the sphere is the Adams--Novikov spectral sequence (see e.g.\ \cite{Ravenel:2004xh})
$$E_{2}^{s,t}\left[BP\right]=\textrm{Ext}^{s,t}_{BP_*BP}\left(BP_*,BP_*\right)\Rightarrow\pi_{t-s}^{S}S^{0}_{(p)}$$
giving rise to a filtration on the stable stems. Within this framework, the complex $e$--invariant introduced by Adams \cite{Adams:1966ys} becomes an invariant of first filtration, since it factors through the 1--line of the spectral sequence. On the other hand, Atiyah, Patodi, and Singer have shown that the $e$--invariant of (the element represented by) a parallelized manifold can be determined analytically using the $\eta$--invariant of the (untwisted) Dirac operator \cite{Atiyah:1975ai}, and this has been the route taken by Deninger and Singhof  to show that the Heisenberg nilmanifolds represent an infinite family of non-trivial elements in the stable stems (more precisely, they represent the image of the $J$--homomorphism, at least away from the prime two) \cite{Deninger:1984zt}.

In order to detect second filtration phenomena, Laures introduced the $f$--invariant, which 
is a follow-up to the $e$--invariant and takes values in the divided congruences between modular forms \cite{Laures:1999sh,Laures:2000bs},
\begin{equation*}
f\colon\pi^{S}_{2r+2}S^{0}\rightarrow\underline{\underline{D}}^{\Gamma_{1}(N)}_{r+2}\otimes{\mathbb{Q/Z}}.
\end{equation*}
Again, there is an analytical interpretation of this invariant involving index theory, albeit on manifolds with corners \cite{Bunke:2008,Bodecker:2008pi}. As shown in \cite{Bodecker:2014vn}, in the situation of a principal circle bundle over a parallelized manifold $M^{2r+1}$, the $f$--invariant can be determined analytically using the $\eta$--invariants of the Dirac operator twisted by powers of the tautological line, and one motivation for this article was to provide a family of tractable examples. Since $\pi^{S}_{4}S^{0}=0$, we directly proceed to lattices of rank two, in which case we have:

\begin{cor}\label{beta_{2/2}}
Let $\Lambda$ be a positive definite even lattice of rank two; then the element $\left[\Gamma_{\Lambda}\backslash G_{\Lambda}\right]\in\pi^{S}_{6}S^{0}\cong\ZZ/2$ is non-trivial if and only if $|\dis\Lambda|$ is odd.
\end{cor}

Alas, all we can establish for lattices of higher rank is the following:

\begin{cor}\label{vanish}
For any positive definite even lattice $\Lambda$ of rank $r>2$, the element $\left[\Gamma_{\Lambda}\backslash G_{\Lambda}\right]\in\pi^{S}_{2r+2}S^{0}$  is in Adams--Novikov filtration $>2$.
\end{cor} 

\begin{rmk*} Note in particular that there are no non-trivial elements of Adams--Novikov filtration $>2$ in $\pi^{S}_{2r+2}S^{0}$ for $r\leq7$; hence Corollary \ref{vanish} becomes an honest vanishing theorem in this range.
\end{rmk*}

The rest of this paper is organized as follows: In Section \ref{nilmanifold}, we briefly set up our notation concerning lattices and introduce the nilmanifolds mentioned above. In Section \ref{representation theory}, we adopt a global point of view on the spinors,  constructing a complete orthonormal basis and describing the Dirac operator. Having this explicit description at our disposal, we then proceed to determine the contributions to the adiabatic limit (Section \ref{adiabatic limit}). Finally, the stable homotopy content is dealt with in Section \ref{stable stems} by combining Theorem 1 with our previous work \cite{Bodecker:2014vn}.


\section{A family of nilmanifolds}\label{nilmanifold}

\subsection{Lattices}
Let $\Lambda$ be a positive definite even lattice of rank $r$, i.e.\ a free  $\ZZ$--module $\Lambda$ of rank $r$ equipped with a nondegenerate symmetric bilinear form $\bil\colon\Lambda\times\Lambda\rightarrow \ZZ$ taking the diagonal to even (non-negative) values; we define the associated quadratic form as $\qu(x)={\textstyle{\frac{1}{2}}}\bil(x,x)$, hence
\begin{equation}\label{bilinear form}
\bil(x,y)=\qu(x+y)-\qu(x)-\qu(y).
\end{equation}
By $\Lambda^{\vee}$ we denote the  the dual lattice, i.e.\
\begin{equation*}
\Lambda^{\vee}=\left\{ v\in\Lambda\otimes_{\ZZ}\RR\colon \bil(v,x)\in\ZZ\ \forall\ x\in\Lambda\right\},
\end{equation*}
where the bilinear form has been extended in the natural way. Furthermore, recall that the discriminant group is the quotient group $\dis(\Lambda)=\Lambda^{\vee}/\Lambda$, and that the quadratic form descends to a well defined map
\begin{equation*}
\overline{\qu}\colon \dis(\Lambda)\rightarrow \QQ/\ZZ.
\end{equation*} 
For any nonzero integer $d$, let $\Lambda_{d}$ be the lattice corresponding to the rescaled form $\bil_{d}=d\cdot\bil$, and $\overline{\qu}_{d}$ the induced quadratic form on $\dis(\Lambda_{d})$.

\subsection{Nilmanifolds associated to lattices}
Observe that, due to \eqref{bilinear form}, the assignment
\begin{equation*}
a\mapsto\left((z',c',b')\mapsto(z'+\bil(a,c')+\qu(a)b',c'+ab',b')\right),
\end{equation*}
where $a\in\Lambda$ and $(z',c',b')\in\ZZ\oplus\Lambda\oplus\ZZ$, gives rise to a group homomorphism $\Lambda\rightarrow \mathrm{Aut}(\ZZ\oplus\Lambda\oplus\ZZ)$. Consequently, we may  form the semidirect product of abelian groups
\begin{equation*}
\Gamma_{\Lambda}=(\ZZ\oplus\Lambda\oplus\ZZ)\rtimes\Lambda;
\end{equation*}
extending naturally, we obtain a connected, simply connected real Lie group
\begin{equation*}
G_{\Lambda}=((\ZZ\oplus\Lambda\oplus\ZZ)\otimes_{\ZZ}\RR)\rtimes(\Lambda\otimes_{\ZZ}\RR).
\end{equation*}
Writing elements of this group as tuples, the group multiplication becomes
\begin{equation*}
\begin{split}
&(z, c,b, a)\cdot(z',c',b', a')\\
=&\left(z+z'+\bil(a, c')+\qu( a)b', c+ c'+ ab',b+b', a+ a'\right).
\end{split}
\end{equation*}

Specifying a basis $\{v_{1},\dots,v_{r}\}$, we may identify $\Lambda\cong\ZZ^{r}\subset\RR^{r}$; let us write $(\bil_{ij})$ for the Gram matrix of this basis, i.e.\ having components $\bil_{ij}=\bil(v_{i},v_{j})$. Furthermore, we define $2r+2$ left invariant vector fields by their action on smooth functions  $f\colon G_{\Lambda}\rightarrow\RR$,
\begin{equation*}
\begin{split}
Zf(g)&=\left.{\textstyle\frac{d}{dt}}\right|_{t=0}f\left(g\cdot(t,0,0,0)\right),\quad C_{j}f(g)=\left.{\textstyle\frac{d}{dt}}\right|_{t=0}f\left(g\cdot(0,tv_{j},0,0)\right),\\
Bf(g)&=\left.{\textstyle\frac{d}{dt}}\right|_{t=0}f\left(g\cdot(0,0,t,0)\right),\quad A_{i}f(g)=\left.{\textstyle\frac{d}{dt}}\right|_{t=0}f\left(g\cdot(0,0,0,tv_{i})\right);
\end{split}
\end{equation*}
clearly, these vector fields form a three-step nilpotent real Lie algebra $\mathfrak{g}_{\Lambda}$ subject to the relations
\begin{equation}
[A_{i},B]=C_{i},\quad [A_{i},C_{j}]=\bil_{ij}Z
\end{equation}
Now let us use boldface letters to denote $r$--tuples of real numbers, so that the identification $\RR^{r}\cong\Lambda\otimes\RR$ may be written as
$${\bf x}=(x_{1},\dots,x_{r})\mapsto\sum_{i=1}^{r}x_{i}v_{i};$$
then we obtain a diffeomorphism $\varphi\colon\RR\times\RR^{r}\times\RR\times\RR^{r}\rightarrow G_{\Lambda}$,
\begin{equation}\label{varphi}
\varphi\colon(z,{\bf{c}},b,{\bf{a}})\mapsto\left(z,{\textstyle{\sum_{j=1}^{r}c_{j}v_{j}}},b,{\textstyle{\sum_{i=1}^{r}a_{i}v_{i}}}\right),
\end{equation}
providing us with global coordinates.
{\em From now on, we  assume that a basis for the fixed lattice $\Lambda$ has been chosen; dropping the reference to the lattice, we simply denote our group by $G$, and tacitly identify it with $\RR^{2r+2}$ by means of the diffeomorphism $\varphi$; for brevity, pullbacks and pushforwards will be omitted from the notation as well, writing e.g.\ $\qu({\bf x})=\sum_{i,j} x_{i}\bil_{ij}x_{j}/2$ for the quadratic form.}
\begin{rmk}
Abstractly, the map $\varphi\colon\RR^{2r+2}\rightarrow G$ should be written as
\begin{equation*}
(z,{\bf{c}},b,{\bf{a}})\mapsto\exp(zZ)\left({\textstyle\prod_{j=1}^{r}}\exp(c_{j}C_{j})\right)\exp(bB)\left({\textstyle\prod_{i=1}^{r}}\exp(a_{i}A_{i})\right),
\end{equation*}
where $\exp\colon\mathfrak{g}\rightarrow G$ is the exponential map; put differently, we describe $G$ in so-called Malcev coordinates (cf.\ e.g.\ \cite{Corwin:1990ve}).
\end{rmk}
 
Forgetting the first coordinate gives rise to a Lie group homomorphism
\begin{equation}\label{quotient}
q\colon G\rightarrow\overline{G}\cong G/\mathcal{Z}(G),
\end{equation}
where $\mathcal{Z}(G)=\left\{\exp\left(zZ\right)\colon z\in\RR\right\}\cong\RR$ is the center of $G$.

Furthermore, the subset $\ZZ\times\ZZ^{r}\times\ZZ\times\ZZ^{r}\subset\RR\times\RR^{r}\times\RR\times\RR^{r}$ corresponds under the map \eqref{varphi} to a discrete subgroup $\Gamma\subset G$; clearly, $\Gamma$ is cocompact and torsion-free.

Putting $\overline{\Gamma}=q(\Gamma)$, we see that $q$ descends to the projection map of a right circle bundle $\Gamma\backslash G\rightarrow \overline{\Gamma}\backslash\overline{G}=:M$. Finally, let $\lambda$ be the complex line bundle associated via the standard representation, i.e.\ having total space
\begin{equation}
E(\lambda)=\left\{[\Gamma g, v]\colon (\Gamma g,v)\in \Gamma\backslash G\times\CC, (\Gamma g\cdot \exp(zZ),v)\sim(\Gamma g, e^{2\pi i z}v)\right\}.
\end{equation}

\subsection{Connections and fiber bundle structure}
By straightforward computation we obtain the following coordinate descriptions of the left invariant vector fields:
\begin{equation*}
\begin{split}
Zf(g)&=\partial_{z}f(g),\quad C_{j}f(g)=\left(\sum_{k}a_{k}\bil_{kj}\partial_{z}+\partial_{c_{j}}\right)f(g),\\
 B f(g)&=\left({\textstyle\frac{1}{2}}\sum_{i,j}a_{i}\bil_{ij}a_{j}\partial_{z}+\sum_{k}a_{k}\partial_{c_{k}}+\partial_{b}\right)f(g),\quad  A_{i}f(g)=\partial_{a_{i}}f(g).
\end{split}
\end{equation*}
The dual coframe fields are then:
\begin{equation*}
\begin{split}
Z^{*}&={\rm d}z-\sum_{k,l}a_{k}\bil_{kl}{\rm d}c_{l}+{\textstyle\frac{1}{2}}\sum_{i,j}a_{i}\bil_{ij}a_{j}{\rm d}b,\\
C^{*}_{j}&={\rm d}c_{j}-a_{j}{\rm d}b,\quad B^{*}={\rm d}b,\quad A^{*}_{i}={\rm d}a_{i}.
\end{split}
\end{equation*}
By construction, these left invariant objects descend to the nilmanifold $\Gamma\backslash G$; by abuse of notation, we still denote the induced vector fields by $Z$, $C_{j}$, $B$, $A_{i}$, and the induced forms still by $Z^{*}$, $C_{j}^{*}$, $B^{*}$, $A_{i}^{*}$.

Using the homomorphism \eqref{quotient}, we obtain $2r+1$ linearly independent vector fields on $M=\overline{\Gamma}\backslash{\overline G}$, viz.
\begin{equation*}
\overline{C}_{j}=q_{*}C_{j},\quad \overline{B}=q_{*}B, \quad \overline{A}_{i}=q_{*}A_{i}.
\end{equation*}
Note that since $\bil$ is symmetric, we can find an orthogonal matrix diagonalizing $(\bil_{ij})$ (over $\RR$); fixing such a matrix $O$, we have $O(\bil_{ij})O^{-1}=\diag(\nu_{1},\dots,\nu_{r})$; now put
$$X_{i}=\sum_{k}O_{ik}A_{k},\quad Y_{j}=[B,X_{j}]=\sum_{l} O_{jl}C_{l},$$
hence
$$[X_{i},Y_{j}]=\sum O_{ik}O_{jl}\bil_{kl}Z=\sum O_{ik}(O^{-1})_{lj}\bil_{kl}Z=\nu_{i}\delta_{ij}Z.$$
Now consider their projections $\overline{X}_{i}=q_{*}X_{i}$, $\overline{Y}_{j}=q_{*}Y_{j}$; the ordered set of vector fields  given by $\{\overline{B},\overline{X}_{1},\overline{Y}_{1},\dots,\overline{X}_{r},\overline{Y}_{r}\}$ determines an oriented global frame for $M$. Declaring this frame to be orthonormal, we obtain a Riemannian metric $\langle\cdot,\cdot\rangle$ on $M$.
\begin{rmk}
The orientation chosen above will be convenient later on; moreover, it does not depend on a choice of orientation for the lattice $\Lambda$.
\end{rmk}
\begin{rmk}
If one is unhappy to use diagonalization over $\RR$, there is an alternative: The commutator relation $[A_{i},C_{j}]=\bil_{ij}Z$ can be used to define a non-degenerate skew bilinear form $\Omega$ on the real vector space $(\Lambda\oplus\Lambda)\otimes\RR$ taking integer values on the lattice $\Lambda\oplus\Lambda$. Thus, at the expense of having to use (potentially) different integral basises for the two summands, we can bring $\Omega$ into the form $\begin{pmatrix}0&\nu'\\ -\nu'&0\end{pmatrix}$, such that $\nu'=\diag(\nu_{1}',\dots,\nu_{r}')$ and the non-zero integers $\nu_{i}'$ form a chain of divisors. While simplifying one set of commutator relations, this basis complicates the other set, i.e.\ those involving $B$, for which  reason we are not going to follow this route any further.
\end{rmk}

On the tangent bundle, we can introduce several connections compatible with the metric $\langle\cdot,\cdot\rangle$: Besides the tautological flat connection annihilating the frame, denoted $\nabla^{TM,0}$, we have  the Levi-Civita connection $\nabla^{TM,LC}$; using the Koszul formula, we see that the non-trivial Christoffel symbols are given by
\begin{equation}\label{Christoffel}
\begin{split}
&\langle\nabla^{TM,LC}_{\overline{X}_{i}}\overline{B},\overline{Y}_{j}\rangle=\langle\nabla^{TM,LC}_{\overline{Y}_{j}}\overline{B},\overline{X}_{i}\rangle=\langle\nabla^{TM,LC}_{\overline{B}}\overline{Y}_{j},\overline{X}_{i}\rangle=\frac{1}{2}\delta_{ij},\\
&\langle\nabla^{TM,LC}_{\overline{X}_{i}}\overline{Y}_{j},\overline{B}\rangle=\langle\nabla^{TM,LC}_{\overline{Y}_{j}}\overline{X}_{i},\overline{B}\rangle=\langle\nabla^{TM,LC}_{\overline{B}}\overline{X}_{i},\overline{Y}_{j}\rangle=-\frac{1}{2}\delta_{ij}.
\end{split}
\end{equation}

Now consider the hyperplane $\RR\times\RR^{r}\times\{0\}\times\RR^{r}\subset\RR\times\RR^{r}\times\RR\times\RR^{r}$;  under the map \eqref{varphi} it corresponds to a normal subgroup $H\subset G$ (isomorphic to a Heisenberg group), giving rise to fiber bundles
\begin{equation}
p\colon\Gamma\backslash G\rightarrow \Gamma\backslash G/H\cong\ZZ\backslash \RR\textrm{  and  }
\overline{p}\colon\overline{\Gamma}\backslash\overline{G}\rightarrow \ZZ\backslash\RR.
\end{equation}
Clearly, our trivialization provides a splitting of the tangent bundle into its horizontal and vertical subbundles,
\begin{equation}\label{splitting}
TM\cong T^{H}M\oplus T^{V}M\cong\overline{p}^{*}T(\ZZ\backslash\RR)\oplus T^{V}M;
\end{equation}
this decomposition is orthogonal with respect to the metric $\langle\cdot,\cdot\rangle$ and the map $\overline{p}$ becomes a Riemannian submersion onto the circle of unit length, each fiber having unit volume.
\begin{rmk} This exhibits another semidirect product description of the group $G$, viz.\ $G\cong\RR\ltimes H$.
\end{rmk}
However, the Levi-Civita connection does not preserve the splitting \eqref{splitting}; following \cite{Bismut:1989dk}, we can construct a metric-compatible connection $\nabla^{TM}$ that does (using the natural projections), resulting in:
\begin{equation}\label{modifiedconnection}
\begin{split}
\langle\nabla^{TM}_{\overline{B}}\overline{X}_{i},\overline{Y}_{j}\rangle&=-\frac{1}{2}\delta_{ij}=-\langle\nabla^{TM}_{\overline{B}}\overline{Y}_{j},\overline{X}_{i}\rangle,\\
\nabla^{TM}_{\overline{X}_{i}}\overline{B}&=0=\nabla^{TM}_{\overline{Y}_{j}}\overline{B},\\
\nabla^{TM}_{\overline{X}_{i}}\overline{Y}_{j}&=0=\nabla^{TM}_{\overline{Y}_{j}}\overline{X}_{i}.
\end{split}
\end{equation}
For future reference, let $S=\nabla^{TM,LC}-\nabla^{TM}$ be the difference tensor; in particular, observe that $S(\overline{X}_{i})\overline{X}_{i}=0=S(\overline{Y}_{j})\overline{Y}_{j}$.

Finally, note that the standard hermitian form on $\CC$ naturally induces a hermitian metric on the tautological line bundle $\lambda$; we equip it with the canonical connection $\nabla^{\lambda}$  (which is unitary by construction).


\section{Twisted spinors on $M$}\label{representation theory}
\subsection{The twisted Dirac operator}
Endowing $M$ with the trivial spin structure obtained from tautologically lifting the trivialized bundle of oriented orthonormal frames, we associate the bundle of untwisted spinors, $\Sigma(M)$. The Levi-Civita connection $\nabla^{TM,LC}$ together with the unitary connection $\nabla^{\lambda}$ on $\lambda$ induces a connection $\nabla^{LC}$ on $\Sigma(M)\otimes\lambda^{d}$, and, having specified a Clifford module structure, the twisted Dirac operator can be defined as $(\eth^{LC}\otimes\lambda^{d}) \psi=\sum_{k=0}^{2r}e_{k}\cdot\nabla^{LC}_{e_{k}}\psi$. On the other hand, the splitting of the tangent bundle into trivialized subbundles induces a decomposition  
\begin{equation*}
\Sigma(M)\cong\overline{p}^{*}\Sigma(\ZZ\backslash\RR)\otimes\Sigma^{V}(M)\cong M\times(\CC\otimes\Sigma_{2r});
\end{equation*}denoting by $\nabla$ the connection on $\overline{p}^{*}\Sigma(\ZZ\backslash\RR)\otimes\Sigma^{V}(M)\otimes\lambda^{d}$ induced by $\nabla^{TM}$ (and  $\nabla^{\lambda}$), one readily checks that the twisted Dirac operator can also be expressed as
\begin{equation*}
\eth^{LC}\otimes\lambda^{d}=:\eth=e_{0}\cdot\nabla_{e_{0}}+\eth^{V},
\end{equation*}
where $\eth^{V}=\sum_{k=1}^{2r}e_{k}\cdot\nabla_{e_{k}}$ restricts to the twisted Dirac operator on each fiber (cf.\ formula (4.26) of \cite{Bismut:1989dk} for the general case).

\begin{rmk}
For an explicit model of the Clifford module, consider the matrices 
\begin{equation*}
\sigma_{x}=\begin{pmatrix}0&-1\\1&0\end{pmatrix},\quad \sigma_{y}=\begin{pmatrix}0&i\\i&0\end{pmatrix}, \quad \sigma_{z}=i\sigma_{x}\sigma_{y}=\begin{pmatrix}1&0\\0&-1\end{pmatrix};
\end{equation*}
then $\sigma_{x}$, $\sigma_{y}$ act naturally on $\CC^{2}$ (from the left), generating the complex Clifford algebra $\CC l_{2}$; clearly, these matrices are antihermitian with respect to the usual hermitian product on $\CC^{2}$. Now we can turn $(\CC^{2})^{\otimes r}$ into a $\CC l_{2r}$ module by letting  $e_{2k-1},e_{2k}\in\{e_{1},\dots,e_{2r}\}$ act as $1^{\otimes(k-1)}\otimes \sigma_{x}\otimes\sigma_{z}^{\otimes(r-k)}$ and $1^{\otimes(k-1)}\otimes\sigma_{y}\otimes\sigma_{z}^{\otimes(r-k)}$, respectively; we denote this irreducible module by $\Sigma_{2r}$. Moreover, under the identification $\CC\otimes(\CC^{2})^{\otimes r}\cong (\CC^{2})^{\otimes r}$, $(-i)\otimes\sigma_{z}^{\otimes r}$ acts compatibly as an additional element $e_{0}$, so that  $(\CC^{2})^{\otimes r}$ also becomes a $\CC l_{2r+1}$ module $\Sigma_{2r+1}$ (on which the complex volume element $i^{r+1}e_{0}\dots e_{2r}$ acts as the identity). 
\end{rmk}

To make the Dirac operator (even) more explicit, we identify sections of the twisted spinor bundle with functions 
\begin{equation}
\psi\colon G\rightarrow \CC\otimes\Sigma_{2r}\otimes\CC\cong\Sigma_{2r+1}\otimes\CC\cong\Sigma_{2r+1}
\end{equation}
that are
\begin{itemize}
\item[(i)] invariant under the left regular action of  $\Gamma$ (i.e.\ descend to $\Gamma\backslash G$),
\item[(ii)] square integrable on a fundamental domain for $\Gamma$,
\item[(iii)] equivariant with respect to the center, 
\begin{equation}\label{center-equivariance}
\psi(g_{0}\cdot \exp(zZ))=e^{-2\pi idz}\psi(g_{0}).
\end{equation}
\end{itemize}
Clearly, the space of spinors is a $G$--module under the right regular action, $g\cdot\psi(g_{0})=\psi(g_{0}g)$, and we have
\begin{equation}\label{Dirac operator}
\eth\psi=e_{0}\cdot \nabla_{e_{0}}\psi+\sum_{k=1}^{r}\left(e_{2k-1}\cdot (X_{k})_{*}+e_{2k}\cdot(Y_{k})_{*}\right)\psi,
\end{equation}
where 
\begin{equation}\label{basederivative}
\nabla_{e_{0}}\psi=(B)_{*}\psi-{\textstyle\frac{1}{4}}\sum_{k=1}^{r}e_{2k-1}e_{2k}\cdot\psi
\end{equation}
and the lower star indicates the derived right regular action.

\subsection{Decomposition of the space of spinors}

In order to decompose the space of spinors, we gather some information on the representations of the group $G$:
 \begin{prop}
The equivalence classes of irreducible unitary representations of $G$ with non-trivial action of the center may be parameterized by the elements $\mu=\delta Z^{*}+\beta B^{*}\in\mathfrak{g}^{*}$ with $\delta\neq0$. Moreover, the corresponding representation may be modelled on $L^{2}(\RR^{r},\CC)$, where the action on $\bar f\colon\RR^{r}\rightarrow\CC$ is given by  
$$(z,{\bf c}, b,{\bf a})\cdot\bar f({\bf x})=e^{2\pi i\delta (z+\bil({\bf x},{\bf c})+\qu({\bf x})b)}e^{2\pi i \beta b}\bar f({\bf x+ a}).$$
\end{prop}
\begin{proof}
Recall that any Lie group $G$ acts on the dual of its Lie algebra, $\mathfrak{g}^{*}$, via the coadjoint action
\begin{equation*}
Ad^{*}_{g}\left(\mu\left(X\right)\right)=\mu\left(Ad_{g^{-1}}(X)\right)
\end{equation*}
and that, in the 1--connected nilpotent case, Kirillov theory sets up a bijection between coadjoint orbits and equivalence classes of irreducible unitary representations (see e.g.\ \cite{Corwin:1990ve}). In the situation at hand, we compute
\begin{equation*}
\begin{split}
Ad_{g^{-1}}Z&=Z,\\
Ad_{g^{-1}}C_{j}&=-\sum_{k}a_{k}\bil_{kj}Z+C_{j},\\
Ad_{g^{-1}}B&={\textstyle\frac{1}{2}}\sum_{i,j}a_{i}\bil_{ij}a_{j}Z-\sum_{k}a_{k}C_{k}+B,\\
Ad_{g^{-1}}A_{i}&=\sum_{k}\bil_{ik}(c_{k}-a_{k}b)Z+bC_{i}+A_{i},
\end{split}
\end{equation*}
from which we deduce
\begin{equation*}
\begin{split}
&Ad_{g}^{*}\left(\delta Z^{*}+\sum_{j}\gamma_{j}C_{j}^{*}+\beta B^{*}+\sum_{i}\alpha_{i}A_{i}^{*}\right)\\
&=\delta Z^{*}+\sum_{j}\left(\gamma_{j}-\sum_{k}a_{k}\bil_{kj}\delta\right)C^{*}_{j}\\
&\quad+\left(\beta-\sum_{k}a_{k}\gamma_{k}+{\textstyle\frac{1}{2}}\sum_{i,j}a_{i}\bil_{ij}a_{j}\delta\right)B^{*}\\
&\quad+\sum_{i}\left(\alpha_{i}+b\gamma_{i}+\sum_{k}\bil_{ik}(c_{k}-a_{k}b)\delta\right)A_{i}^{*}.
\end{split}
\end{equation*}
Orbit representatives for $\delta\neq0$ may be taken as $\{\mu=\delta Z^{*}+\beta B^{*}, \delta\neq0\}$, and the abelian ideal $\mathfrak{p}=\spr\{Z,C_{1},\dots,C_{r},B\}$ is a polarizing subalgebra for $\mu$, i.e.\ it is maximally isotropic  for the skew bilinear form $\mu\left([\cdot,\cdot]\right)$. Therefore, putting $\chi_{\mu}(\exp(U))=e^{2\pi i\mu(U)}$ for $U\in{\mathfrak p}$, we obtain a one-dimensional representation of $P=\exp(\mathfrak{p})$. This representation then induces a representation of $G$, viz.\ we consider the space of $P$--equivariant functions $f\colon G\rightarrow\CC$, $f(pg)=\chi_{\mu}(p)f(g)$, that are square integrable on $P\backslash G$, and let $G$ act as follows: $g\cdot f(g_{0})=f(g_{0}g)$. Since $(z,{\bf c},b,{\bf a})=(z,{\bf c},b,0)\cdot(0,0,0,{\bf a})$, such functions $f$ may be identified with square integrable functions $\bar f\colon \RR^{r}\cong P\backslash G\rightarrow \CC$, and due to
\begin{equation*}
\begin{split}
&\quad(0,0,0,{\bf x})\cdot(z,{\bf c},b,{\bf a})\\
&=(z+\bil({\bf x,c})+\qu({\bf x})b,{\bf c+x}b,b,{\bf x+a})\\
&=(z+\bil({\bf x,c})+\qu({\bf x})b,{\bf c+x}b,b,0)\cdot(0,0,0,{\bf x+a}),
\end{split}
\end{equation*}
the right regular action on $f\in L^{2}(G,\CC)$ translates into the action
\begin{equation}
(z,{\bf c},b,{\bf a})\cdot\bar f({\bf x})=e^{2\pi i\delta (z+\bil({\bf x},{\bf c})+\qu({\bf x})b)}e^{2\pi i \beta b}\bar f({\bf x+ a})
\end{equation}
on $\bar f\in L^{2}(\RR^{r},\CC)$.
\end{proof}
Due to the equivariance requirement \eqref{center-equivariance}, the relevant situation is that of $\delta=-d$ for some non-zero integer $d$, and the straightforward computation of the derived right regular action yields:
\begin{equation}\label{derived}
\begin{split}
(A_{i})_{*}\bar f ({\bf x})&=\frac{\partial}{\partial x_{i}}\bar f({\bf x}),\quad (B)_{*}\bar f ({\bf x})=2\pi i(\beta-d\qu({\bf x})) \bar f({\bf x}),\\
(C_{j})_{*}\bar f ({\bf x})&=-2\pi i d\sum_{i}x_{i}\bil_{ij}\bar f({\bf x}),\quad (Z)_{*}\bar f({\bf x})=-2\pi i d \bar f({\bf x}).
\end{split}
\end{equation}
It is well-known that the Hermite functions form a complete orthonormal basis for the space of square integrable functions; we find it convenient to describe them as follows: Using our previously defined tangent vector fields $X_{i}$, $Y_{j}$, we introduce operators
$$\mathcal{A}_{k}=\left(X_{k}+i\sign(d)Y_{k}\right)_{*},\quad {\mathcal{A}_{k}}^{\dagger}=\left(i\sign(d)Y_{k}-X_{k}\right)_{*},$$
which are (rescaled versions of) annihilation and creation operators, respectively, i.e.\ they satisfy
$$\left[\mathcal{A}_{k},{\mathcal{A}_{l}}^{\dagger} \right]=4\pi|d|\nu_{k}\delta_{kl},\quad \left[\mathcal{A}_{k},\mathcal{A}_{l} \right]=0=\left[{\mathcal{A}_{k}}^{\dagger},{\mathcal{A}_{l}}^{\dagger} \right]$$
and therefore
$$\left[{\mathcal{A}_{k}}^{\dagger}\mathcal{A}_{k},{\mathcal{A}_{l}}^{\dagger} \right]=4\pi|d|\nu_{k}\delta_{kl}{\mathcal{A}_{k}}^{\dagger},\quad \left[{\mathcal{A}_{k}}^{\dagger}\mathcal{A}_{k},\mathcal{A}_{l} \right]=-4\pi|d|\nu_{k}\delta_{kl}\mathcal{A}_{k}.$$
Consider the smooth function $\phi_{0}({\bf x})=\exp(-2\pi|d|\qu({\bf x}))$, which is annihilated by all of the $\mathcal{A}_{k}$; then, for each $r$--tuple of non-negative integers ${\bf n}=(n_{1},\dots,n_{r})\in\NN_{0}^{r}$, applying each creation operator ${\mathcal{A}_{k}}^{\dagger}$ $n_{k}$ times, we obtain a smooth function $\phi_{{\bf n}}$ that is a simultaneous eigenvector for all `number operators', i.e.\ ${\mathcal{A}_{k}}^{\dagger}\mathcal{A}_{k}\phi_{{\bf n}}=4\pi|d|\nu_{k}n_{k}\phi_{{\bf n}}$. Normalizing with respect to the $L^{2}$ scalar product, i.e.\ putting $h_{{\bf n}}({\bf x})=\phi_{{\bf n}}({\bf x})/\sqrt{||\phi_{{\bf n}}||^{2}}$, we clearly have $\langle h_{{\bf n}}, h_{{\bf n'}}\rangle_{L^{2}}=\delta_{{\bf n\, n'}}$, and one easily proves completeness as well.

Next, we decompose the spinor module $\Sigma_{2r}$ using simultaneous eigenvectors (of unit norm) of the $r$ involutions $\omega_{k}=ie_{2k-1}e_{2k}$; letting $s_{k}\in\{\pm1\}$ be the eigenvalue of $\omega_{k}$, we find it convenient to label the elements of such a fixed basis by their `chirality vector' ${\bf s}=(s_{1},\dots,s_{r})\in\{\pm1\}^{r}$.

Finally, given an even positive definite lattice $\Lambda$ and a non-zero integer $d$, let $R_{d}$ denote a fixed choice of lift of the discriminant $\dis\Lambda_{d}$, i.e.\ a subset $R_{d}\subset(\Lambda_{d})^{\vee}$ containing precisely one representative $\tilde\rho$  for each class $\rho\in\dis\Lambda_{d}$.

Given an element
$$({\bf n}, {\bf s}, \tilde\rho, \overline{\beta})\in \NN^{r}_{0}\times\{\pm1\}^{r}\times R_{d}\times\ZZ$$
we can construct a smooth twisted spinor as follows: 
We take the function $h_{{\bf n}}({\bf x})$, considered as a vector in the $G$ representation determined by the form $\mu=-dZ^{*}+(\overline{\beta}+\qu_{d}(\tilde\rho))B^{*}$, identify it with a $P$--equivariant function on $G$, say $$\Phi_{({\bf n},\tilde\rho,\overline{\beta})}(z,{\bf c},b,{\bf a})=e^{-2\pi idz}e^{2\pi i (\overline{\beta}+\qu_{d}(\tilde\rho))b}h_{{\bf n}}({\bf a});$$
taking the left translate by an element $(0,0,0,{\bf k}+\tilde\rho)$ then results in a function
\begin{equation*}
\begin{split}
&\quad \left(L_{(0,0,0,{\bf k}+\tilde\rho)}^{*}\Phi_{({\bf n},\tilde\rho,\overline{\beta})}\right)(z,{\bf c},b,{\bf a})\\
&=e^{-2\pi id(z+\bil({\bf k}+\tilde\rho,{\bf c})+\qu({\bf k}+\tilde\rho)b)}e^{2\pi i (\overline{\beta}+\qu_{d}(\tilde\rho))b}h_{{\bf n}}({\bf a+k}+\tilde\rho)\\
&=e^{-2\pi id(z+\bil({\bf k}+\tilde\rho,{\bf c})+\qu({\bf k})b+\bil({\bf k},\tilde{\rho})b)}e^{2\pi i \overline{\beta}b}h_{{\bf n}}({\bf a+k}+\tilde\rho).
\end{split}
\end{equation*}
This function does not descend to the quotient $\Gamma\backslash G$, but this can be remedied easily by summing over translates; tensoring with the  basis vector for the Clifford module of the specified chirality, we then obtain a smooth spinor
\begin{equation}\label{ON-spinor}
\psi_{({\bf n},{\bf s},\tilde\rho,{\overline\beta})}(z,{\bf c},b,{\bf a}):=\sum_{{\bf k}\in\ZZ^{r}\cong\Lambda}\left(L_{(0,0,0,{\bf k}+\tilde\rho)}^{*}\Phi_{({\bf n},\tilde\rho,\overline{\beta})}\right)(z,{\bf c},b,{\bf a})\otimes e_{{\bf s}}
\end{equation}
of unit norm (with respect to the $L^{2}$ scalar product on spinors). Conversely, we have:

\begin{prop}
For any non-zero integer $d$ and fixed choice of $R_{d}$, the space of twisted spinors can be decomposed \textup{(}as a $G$--$\CC l_{2r+1}$ bimodule\textup{)} into  a countable sum of Hilbert spaces
$$L^{2}(M,\Sigma(M)\otimes\lambda^{d})\cong{\bigoplus}_{\tilde\rho\in R_{d},\ \beta-\qu_{d}(\tilde\rho)\in\ZZ}\mathcal{H}_{\tilde\rho,\beta}$$
where for each $\tilde\rho\in R_{d}$, we have $\mathcal{H}_{\tilde\rho,\beta}\cong L^{2}(\RR^{r},\CC)\otimes\Sigma_{2r+1}$, the $G$--action on $L^{2}(\RR^{r},\CC)$ being determined by $\mu=-dZ^{*}+\beta B^{*}$.
\end{prop}

\begin{proof} Consider an arbitrary spinor $\psi\colon G\rightarrow \Sigma_{2r+1}$; our assumptions imply that $\psi$ admits a Fourier expansion of the form
\begin{equation*}
\begin{split}
\psi((z,{\bf c},b,{\bf a}))&=\sum_{\beta'\in\ZZ,\gamma'\in\ZZ^{r}}\hat\psi_{\gamma',\beta'}'({\bf a})e^{2\pi i(-dz+\beta' b+\sum_{j}\gamma_{j}' c_{j})}\\
&=\sum_{\beta'\in\ZZ,\gamma\in(\Lambda_{d})^{\vee}}\hat\psi_{\gamma,\beta'}({\bf a})e^{2\pi i(-dz+\beta' b-\bil_{d}(\gamma,{\bf c}))}
\end{split}
\end{equation*}
where we interpreted $-\bil_{d}$ as a bijection from the set  $(\Lambda_{d})^{\vee}$ to $\ZZ^{r}$, the primed and unprimed coefficients being related in the obvious way. 
Furthermore, the invariance requirement $\psi((0,0,0,{\bf k})\cdot(z,{\bf c},b,{\bf a}))=\psi((z,{\bf c},b,{\bf a}))$ for ${\bf k}\in\ZZ^{r}$ then imposes a constraint on the Fourier coefficients, viz.\
$$\hat\psi_{\gamma,\beta'}({\bf a+k})=\hat\psi_{\gamma+{\bf k},\beta'-\bil_{d}(\gamma,{\bf k})-\qu_{d}({\bf k})}({\bf a}),$$
where  any ${\bf k}$ occuring in the subscript is to be understood as an element in $\Lambda_{d}\subset(\Lambda_{d})^{\vee}$. Hence for fixed $\beta'\in\ZZ$, $\tilde\rho\in R_{d}$, the Fourier coefficient $\hat\psi_{\tilde\rho,\beta'}$ gives rise to a function in $L^{2}(\RR^{r},\Sigma_{2r+1})$; by comparison to the construction leading to \eqref{ON-spinor}, the claim now follows. 
\end{proof}

Using the isomorphism of the previous proposition, we are now going to write $\bar{\psi}_{({\bf n},{\bf s},\tilde\rho,{\overline\beta})}({\bf x})\in L^{2}(\RR^{r},\CC)\otimes\Sigma_{2r+1}$ for the spinor constructed in \eqref{ON-spinor} (which is consistent with the notation in the expressions \eqref{derived} for the derived action); then the covariant derivative \eqref{basederivative} of such a spinor is given by
\begin{equation*}
\nabla_{e_{0}}\bar\psi_{({\bf n},{\bf s},\tilde\rho,\overline{\beta})}({\bf x})=2\pi i\left(\beta-\left(\sign(d)\qu_{|d|}({\bf x})-{\textstyle \frac{1}{8\pi}\sum_{k} s_{k}}\right)\right)\bar\psi_{({\bf n},{\bf s},\tilde\rho,\overline{\beta})}({\bf x}),
\end{equation*}
where $\beta=\overline{\beta}+\qu_{d}(\tilde\rho)$. Note that the quadratic term can be expressed in terms of annihilation and creation operators: Using the diagonalizing property of the matrix $O$
\begin{equation*}
\qu({\bf x})=\frac{1}{2}\sum_{i,j}x_{i}\bil_{ij}x_{j}=\frac{1}{2}\sum_{l,m,n}\nu_{l}(O_{lm}x_{m})(O_{ln}x_{n})
\end{equation*}
and the identity
\begin{equation*}
\left(\mathcal{A}_{i}+{\mathcal{A}_{i}}^{\dagger}\right)\bar\psi_{({\bf n},{\bf s},\tilde\rho,\overline{\beta})}({\bf x})=4\pi|d|\sum_{j}\nu_{i}O_{ij}x_{j}\bar\psi_{({\bf n},{\bf s},\tilde\rho,\overline{\beta})}({\bf x}),
\end{equation*}
it follows that
\begin{equation}\label{Q decomposed}
\qu({\bf x})\bar\psi_{({\bf n},{\bf s},\tilde\rho,\overline{\beta})}({\bf x})=\frac{1}{2}\sum_{k}\frac{\left(\mathcal{A}_{k}+{\mathcal{A}_{k}}^{\dagger}\right)^{2}}{\left(4\pi |d|\right)^{2}\nu_{k}}\ \bar\psi_{({\bf n},{\bf s},\tilde\rho,\overline{\beta})}({\bf x}).
\end{equation}
Here is another important observation: The vertical spin Laplacian is given by 
\begin{equation}\label{vertical laplacian}
(\eth^{V})^{2}=\sum_{k}\left({\mathcal{A}_{k}}^{\dagger}\mathcal{A}_{k}+2\pi|d|\nu_{k}(1-\sign(d)\omega_{k})\right),
\end{equation}
so that zero eigenvalues necessarily have  chirality ${\bf s}=\sign(d)(1,\dots,1)$.


\section{The adiabatic limit}\label{adiabatic limit}
Let us briefly recall some notations concerning the adiabtic limit (cf.\ \cite{Bismut:1989dk,Dai:1991os,Berline:2004pb}): First of all, for any vector bundle $E$  over the total space of the fiber bundle ${\overline p}\colon M\rightarrow \ZZ\backslash\RR$, one can associate a vector bundle  $\tilde E$ over the base $\ZZ\backslash\RR$, the fiber of which is the space of smooth sections of the restriction of $E$ to the fiber of $\overline p$. In particular, we can consider the vertical spinor bundle $\Sigma^{V}M\otimes\lambda^{d}$; its associated bundle will be denoted by  $H_{\infty}$; the connection $\nabla$ on $\Sigma^{V}M\otimes\lambda^{d}$ induces a connection $\tilde\nabla$ on $H_{\infty}$ (since a section $\tilde\psi$ of $H_{\infty}$ may be identified with  a section $\psi$ of $\Sigma^{V}M\otimes\lambda^{d}$, we may put $\tilde\nabla_{{\overline p}_{*}\overline{B}}\tilde\psi=\nabla_{\overline B}\psi$.)

Moreover, this bundle is $\ZZ/2$--graded (by the action of $\omega_{1}\cdot...\cdot\omega_{r}$), and we can define the so-called Levi-Civita superconnection $\mathbb{A}_{t}$ associated to the family of vertical Dirac operators $\eth^{V}$; since the bundle is flat, and since $\tilde\nabla$ is unitary (recall that $S(\overline{X}_{i})\overline{X}_{i}=0=S(\overline{Y}_{j})\overline{Y}_{j}$), this superconnection simplifies to  $\mathbb{A}_{t}=t^{1/2}\eth^{V}+\tilde\nabla$.

Applying the index theorem for families to $\eth^{V}$, and making use of the Weitzenb\"ock formula \eqref{vertical laplacian}, we conclude that the kernel of $\eth^{V}$ forms a vector bundle of rank $|d^{r}\det(\bil_{ij})|=|\dis\Lambda_{d}|$. Therefore the $\eta$--form of the Levi-Civita superconnection is well-defined \cite{Berline:2004pb},
\begin{equation}
\hat\eta=\int_{0}^{\infty}\tr_{s}\left((\partial_{t}\mathbb{A}_{t})e^{-\mathbb{A}_{t}^{2}}\right)dt.
\end{equation}
In our situation, we actually have:
\begin{prop}\label{eta hat vanishes}
The $\eta$--form associated to the family $\eth^{V}$ vanishes identically.
\end{prop}
\begin{proof}
For dimensional reasons, it suffices to consider the one-form part, which is given by
$$\hat\eta=-\frac{1}{2}\int_{0}^{\infty}\tr_{s}\left(\left(\eth^{V}\tilde\nabla \eth^{V}+\eth^{V}\eth^{V}\tilde\nabla\right) e^{-t(\eth^{V})^{2}}\right)dt;$$
we claim that the integrand is already trivial. In order to see this, note that the space of eigenvectors of $(\eth^{V})^{2}$ for a given eigenvalue may be decomposed into $\eth^{V}$--invariant subspaces of dimension $\leq2^{r}$ (equality holding unless the number vector ${\bf n}$ contains a zero component); while $\nabla_{e_{0}}$ does not preserve these subspaces, by \eqref{Q decomposed} its action on all non-zero eigenvectors in this particular subspace {\em projects} to a multiplication action by the same number. Since $\tilde\nabla=({\textup d}b)\tilde\nabla_{e_{0}}$, and ${\textup d} b$ anticommutes with $\eth^{V}$, the claim follows (in fact, this explicit description shows that even more cancellation takes place).
\end{proof}

Now let $P^{0}$ be the projection onto the kernel of $\eth^{V}$; then we have an operator $\bar\eth^{B}=e_{0}\cdot P^{0}\nabla_{e_{0}}P^{0}$ which may be considered as a Dirac operator on the base circle twisted by the kernel of the vertical Dirac operator.
\begin{prop}\label{basic eta}
The reduced $\eta$--invariant of the twisted Dirac operator on the base circle is given by
\[\frac{\eta\left(\bar\eth^{B}\right) +\dim\ker\left(\bar\eth^{B}\right)}{2}\equiv \sign(d)^{r}\sum_{\rho\in\dis\Lambda_{d}}\left({\textstyle\frac{1}{2}}-\overline{\qu}_{d}(\rho)\right)\mod\ZZ.\]
\end{prop}
\begin{proof}
On the basis for the kernel of $\eth^{V}$, $P^{0}\nabla_{e_{0}}$ simply acts as multiplication by $2\pi i (\overline{\beta}+\qu_{d}(\tilde\rho))$, hence the spectrum is given by
\[\mathrm{spec}(\bar\eth^{B})=\left\{2\pi \sign(d)^{r}(k+\qu_{d}(\tilde\rho))\colon k\in\ZZ, \tilde\rho\in R_{d}\right\}.\]
Using  the Hurwitz zeta function (defined by $\zeta(s,x)=\sum_{n\geq0}(n+x)^{-s}$ for $x>0$ and $\Re(s)>1$), which satisfies $\zeta(0,x)=\frac{1}{2}-x$, the claim follows.
\end{proof}

Now we can put everything together: For $\epsilon>0$, let $\langle\cdot,\cdot\rangle_{\epsilon}$ be the metric obtained by rescaling the metric on the base by a factor $\epsilon^{-1}$, i.e.\ the one for which $\{\epsilon^{1/2}\overline{B},\overline{X}_{1},\overline{Y}_{1},\dots,\overline{X}_{r},\overline{Y}_{r}\}$ becomes an  orthonormal frame, and let $\eth_{\epsilon}=\epsilon^{1/2}e_{0}\cdot\nabla_{e_{0}}+\eth^{V}$ be the corresponding Dirac operator (still twisted by the $d^{th}$ power of $\lambda$). The reduced $\eta$--invariant of $\eth_{\epsilon}$ admits a limit  (as $\epsilon\rightarrow0$); moreover, since the kernel of $\eth^{V}$ forms a vecor bundle, we may apply \cite[Theorem 0.1']{Dai:1991os}: For the one-dimensional base manifold $\ZZ\backslash\RR$ Dai's formula can be stated as:
\begin{equation*}\label{Dai's formula}
\lim_{\epsilon\rightarrow0}\frac{\eta(\eth_{\epsilon})+\dim\ker(\eth_{\epsilon})}{2}\equiv\frac{\eta\left(\bar\eth^{B}\right) +\dim\ker\left(\bar\eth^{B}\right)}{2}+\frac{1}{2\pi i}\int_{\ZZ\backslash\RR}\hat\eta\mod\ZZ.
\end{equation*}
Thus, Theorem 1 follows from the previous propositions.


\section{Stable homotopy content}\label{stable stems}

\subsection{The $f$--invariant of $[\Gamma\backslash G]$}
By the Atiyah--Patodi--Singer index theorem, the variation of reduced $\eta$--invariants is local; in particular, we can consider the family of Levi-Civita connections $\nabla^{TM,LC,\epsilon}$ associated to the metrics $\langle\cdot,\cdot\rangle_{\epsilon}$ for $\epsilon>0$; moreover, on $M$ we have the tautological flat connection $\nabla^{TM,0}$ annihilating the global frame, hence
\begin{equation*}
\begin{split}
&\quad\frac{\eta(\eth_{\epsilon})+\dim\ker(\eth_{\epsilon})}{2}-\int_{M}cs(\nabla^{TM,LC,\epsilon},\nabla^{TM,0})\\
&\equiv\frac{\eta(\eth)+\dim\ker(\eth)}{2}-\int_{M}cs(\nabla^{TM,LC},\nabla^{TM,0})\mod\ZZ,
\end{split}
\end{equation*}
where $cs$ is shorthand for the Chern--Simons form transgressing the index densities  expressed in terms of the connections on $TM$ (keeping $\nabla^{\lambda}$ fixed). (Here, we use the same sign convention as in \cite{Bodecker:2014vn}.) 

By the results of Bismut and Cheeger, the family of connections $\nabla^{TM,LC,\epsilon}$  approaches a well-defined connection in the limit $\epsilon\rightarrow0$. In our situation, we actually have:
\begin{prop}\label{Chern--Simons vanishes}
$\lim_{\epsilon\rightarrow0}cs(\nabla^{TM,LC,\epsilon},\nabla^{TM,0})=0$.
\end{prop}
\begin{proof}
Since the Chern--Simons form depends continuously on the connection, we may directly compute it for the limiting connection $\nabla^{TM,LC,0}$: Consider the one-forms $\theta^{i}_{j}=\langle\nabla^{TM,LC,0}e_{j},e_{i}\rangle$ (which are globally defined, since $\langle\nabla^{TM,0}e_{j},e_{i}\rangle=0$), combined into a matrix $\theta$. By direct inspection, the matrices  $\theta^{2l+1}({\textup d}\theta)^{l'}$ (where wedge products of forms are omitted) become strictly triangular, hence traceless, provided that $l+l'>0$. Since the Chern--Simons form is built from such traces, the claim follows. 
\end{proof}

Thus, the $f$--invariant is given by the following arithmetic formula:
\begin{lem}\label{lemma}
At a given level $N>1$, the $f$--invariant is given by
\begin{multline*}
f[\Gamma\backslash G]\equiv-\sum_{n\geq1}\left(\sum_{d|n}\left(\zeta_{N}^{-n/d}+(-1)^{r}\zeta_{N}^{n/d}\right)\sum_{\rho\in\dis\Lambda_{d}}{\overline\qu}_{d}(\rho)\right)q^{n}\\
\mod\underline{\underline{D}}_{r+2}^{\Gamma_{1}(N)}+\ZZ^{\Gamma_{1}(N)}[\![q]\!].
\end{multline*}
\end{lem}
\begin{proof}
By Proposition \ref{Chern--Simons vanishes}, we may simply plug the expression obtained from Theorem 1 (which clearly lies in $\QQ/\ZZ$) into the main theorem of \cite{Bodecker:2014vn}; since 
$$\sign(d)^{r}\sum_{\rho\in\dis\Lambda_{d}}\left({\textstyle\frac{1}{2}}-\overline{\qu}_{d}(\rho)\right)\equiv \sign(d)^{r+1}\sum_{\rho\in\dis\Lambda_{d}}\left({\textstyle\frac{1}{2}}-\overline{\qu}_{|d|}(\rho)\right)\mod\ZZ,$$
and $|\dis{\Lambda_{d}}|/2\equiv d^{r+1}|\dis\Lambda|/2$, the result follows from the fact that
$$\widehat{G}_{k}^{(N)}(\tau)=c_{k}-\sum_{n\geq1}\left(\sum_{d|n}\left(\zeta_{N}^{-n/d}+(-1)^{k}\zeta_{N}^{n/d}\right)d^{k-1}\right)q^{n}$$
is a modular form of weight $k$ for the congruence subgroup $\Gamma_{1}(N)$.
\end{proof}

\subsection{Proof of the Corollaries}

Having established Lemma \ref{lemma}, we now construct an explicit rational polynomial lifting  the map
\[d\mapsto\sum_{\rho\in\dis\Lambda_{d}}\overline{\qu}_{d}(\rho).\] The existence of the Smith normal form for the bilinear form, i.e.\ the existence of invertible integral matrices $S,T$ such that $S(\bil_{ij})T=\diag(d_{1},\dots,d_{r})$ where $d_{1}|d_{2}|\dots|d_{r}$ forms a chain of divisors, implies that, for fixed $d$,  there is an integral basis $\{u_{1},\dots,u_{r}\}$ for $(\Lambda_{d})^{\vee}$ such that $\{dd_{1}u_{1},\dots,dd_{r}u_{r}\}$ is a basis for $\Lambda$. Hence, for $r=2$, we may compute
\begin{equation*}
\begin{split}
\sum_{\rho\in\dis\Lambda_{d}}\overline{\qu}_{d}(\rho)&\equiv\sum_{k_{1}=1}^{dd_{1}}\sum_{k_{2}=1}^{dd_{2}}\qu_{d}(k_{1}u_{1}+k_{2}u_{2})\\
&=\sum_{k_{1}=1}^{dd_{1}}\sum_{k_{2}=1}^{dd_{2}}\left(\qu_{d}(k_{1}u_{1})+\qu_{d}(k_{2}u_{2})+\bil_{d}(k_{1}u_{1},k_{2}u_{2})\right)\\
&= d^{2}d_{2}\qu(u_{1})\frac{dd_{1}(dd_{1}+1)(2dd_{1}+1)}{6}\\&\quad+d^{2}d_{1}\qu(u_{2})\frac{dd_{2}(dd_{2}+1)(2dd_{2}+1)}{6}\\&\quad+d\frac{dd_{1}(dd_{1}+1)}{2}\frac{dd_{2}(dd_{2}+1)}{2}\bil(u_{1},u_{2})\\
&=\frac{d^{3}}{3}d_{1}d_{2}\qu(dd_{1}u_{1})+\frac{d^{2}}{2}d_{2}\qu(dd_{1}u_{1})+\frac{d}{6}d_{1}d_{2}\qu(du_{1})\\
&\quad+\frac{d^{3}}{3}d_{1}d_{2}\qu(dd_{2}u_{2})+\frac{d^{2}}{2}d_{1}\qu(dd_{2}u_{2})+\frac{d}{6}d_{1}d_{2}\qu(du_{2})\\
&\quad+\frac{d^{3}d_{1}d_{2}+d^{2}(d_{1}+d_{2})+d}{4}\bil(dd_{1}u_{1},dd_{2}u_{2}).
\end{split}
\end{equation*}
Furthermore, we have
\[\bil(dd_{1}u_{1},dd_{2}u_{2})\equiv|\dis\Lambda|=|d_{1}d_{2}|\equiv d_{1}\mod 2,\]
since $d_{1}$ is the greatest common divisor of the even matrix $(\bil_{ij})$. In particular, it follows that $d_{1}d_{2}\qu(du_{1})$ and $d_{1}d_{2}\qu(du_{2})$ are integers, hence
\[\sum_{\rho\in\dis\Lambda_{d}}\overline{\qu}_{d}(\rho)\equiv\alpha'd^{3}+\frac{|\dis\Lambda|}{4}d\mod\ZZ\]
for some $\alpha'\in\QQ$ satisfying $12\alpha'\in\ZZ$.  But
\begin{equation}\label{top form}
\sum_{n\geq1}\left(\sum_{d|n}\left(\zeta_{N}^{-n/d}+(-1)^{r}\zeta_{N}^{n/d}\right)d^{r+1}\right)q^{n}\equiv0\mod\underline{\underline{D}}_{r+2}^{\Gamma_{1}(N)},
\end{equation}
therefore we conclude
$$f[\Gamma\backslash G]\equiv|\dis\Lambda|\left(\frac{1}{4}\sum_{n\geq1}\left(\sum_{d|n}\left(\zeta_{N}^{-n/d}+\zeta_{N}^{n/d}\right)d\right)q^{n}\right)\mod\underline{\underline{D}}_{4}^{\Gamma_{1}(N)},$$
where the term in parenthesis indeed represents the $f$--invariant of the non-trivial element $\nu^{2}\in\pi^{S}_{6}S^{0}\cong\ZZ/2$ (see e.g.\ \cite{Bodecker:2008pi}), proving Corollary \ref{beta_{2/2}}.

Corollary \ref{vanish} follows similarly: For arbitrary $r\geq1$, we can use the aforementioned (choice of) basis $\{u_{1},\dots,u_{r}\}$ to produce a description of the form
\[\sum_{\rho\in\dis\Lambda_{d}}\overline{\qu}_{d}(\rho)\equiv\alpha d^{r+1}+\beta d^{r}+ \gamma d^{r-1}\mod\ZZ,\]
where $2\beta\in\ZZ$ and $12\gamma\in\ZZ$ by elementary number theory; in particular, this implies 
$\alpha d^{r+1}+\beta d^{r}+ \gamma d^{r-1}\equiv(\alpha+\beta+\gamma)d^{r+1}\mod\ZZ$ for $r\geq3$, hence in this situation,  the $f$--invariant is clearly trivial.

\bibliography{refbib_edited}
\end{document}